\documentclass[12pt, final]{amsart}

\usepackage{geometry}
\usepackage{multirow}
\usepackage{graphicx}
\usepackage{float}
\usepackage{amssymb}
\usepackage{verbatim} 
\usepackage{amsmath}
\usepackage{amsthm}
\usepackage{color}
\usepackage{epstopdf}
\usepackage{epsfig}
\usepackage[all]{xy}
\usepackage{enumerate}
\usepackage{hyperref}

\newcommand{\vol}{{\rm vol}}
\newtheorem{defi}{Definition}
\newtheorem{thm}{Theorem}
\newtheorem{prop}{Proposition}

\renewcommand{\S}{\mathbb{S}}

\begin{document}
\title[Area minimizing unit vector fields and minimally immersed Klein Bottles]{Area minimizing unit vector fields on antipodally punctured unit 2-sphere and minimally immersed Klein Bottles}

\author{Fabiano Brito$^1$}
\author{Jackeline Conrado$^2$}
\author{Adriana Nicoli$^3$}
\author{Icaro Gon\c calves$^4$}

\thanks{The second and third was financed by Coordena\c c\~ao de Aperfei\c coamento de Pessoal de N\'ivel Superior- Brasil (CAPES) - Finance Code 001.}

\address{Centro de Matematica, Computa\c c\~{a}o e Cogni\c c\~{a}o, 
Universidade Federal do ABC, Santo Andr\'{e}, 09210-170, Brazil.}
\email{fabiano.brito@ufabc.edu.br, icaro.goncalves@ufabc.edu.br}

\address{Dpto. de Matem\'{a}tica, Instituto de Matem\'{a}tica e Estat\'{i}stica, 
Universidade de S\~{a}o Paulo, R. do Mat\~{a}o 1010, S\~{a}o Paulo-SP,
05508-900, Brazil.}
\email{avnicoli@usp.br, jconrado@usp.br}

\subjclass[2019]{}

\begin{abstract} 
We provide a lower value for the volume of a unit vector field tangent to an antipodally Euclidean sphere $\mathbb{S}^2$ depending on the length of an ellipse determined by the indexes of its singularities. In addition, for minimizing vector fields having specific pair of indexes, we show that their image coincides with the image of minimally immersed Klein bottles.    
\end{abstract}

\maketitle
\section{Introduction and main results}
We study, in this paper, the relation between indices and total area of unit vector fields defined on antipodally punctured two dimensional sphere $\mathbb{S}^2$.

In the first result, we establish sharp lower bounds for the total area of unit vector fields on $\mathbb{S}^2 \backslash \{N,S\}$, where $N$ and $S$ are antipodal points of $\mathbb{S}^2$.

\begin{thm}\label{Main1} Let $\vec{v}$ be a unit vector field defined on $M = \mathbb{S}^{2} \backslash \left\{N,S\right\}$. If $k = \sup\left\{ I_{\vec{v}}(N), I_{\vec{v}}(S)\right\}$, then
	\[\vol(\vec{v}) \geq \pi L(\xi_k), \]
	where $L(\xi_k)$ is the length of the ellipse $\frac{x^2}{k^2}+\frac{y^2}{(k - 2)^2} = 1$ with $k > 2$ and $I_{\vec{v}}(P)$ stands for the Poincar\'e index of $\vec{v}$ around $P$.
\end{thm}

This is a natural extension of the theorem proved in \cite{BCJ} by P. Chacon, D. Johnson and the first author. In \cite{BCJ}, a general lower-bound for area of unit vector fields in $\mathbb{S}^2 \backslash \{N,S\}$ is established. It turns out to be the area of north-south vector field with both indices equal to $1$. In this context, our first theorem provides certain lower bounds for each class of indice.

We also exhibit minimizing vector fields $\vec{v}_k$ within each index class. These fields have areas given essentially by the length of ellipses depending just on the indices in $N$ and $S$.

In our second theorem, we prove that the images of the minimizing unit vector fields can be seen as the images of minimally immersed Klein bottles when the indices in $N$ and $S$ are even non zero natural numbers. 

\begin{thm}\label{Main2} 
	For every $k$, the set $\overline{\vec{v}_k(\mathbb{S}^2\backslash\{N,S\})}$ is the image of a minimal smooth immersed Klein bottle in $\mathbb{RP}^3 \cong T^1\S^2$.
\end{thm}

The second result is based on the idea of studying the closure in $T^1(\mathbb{S}^2)$ of the images of unit vector fields in  $\mathbb{S}^2\backslash\{N,S\}$. This was first studied on \cite{BorGilM}, where the authors prove there that the case ($2,0$) provides a totally geodesic real projective plane $\mathbb{R}P^2$ minimally embedded in $T^1(\mathbb{S}^2)$.

\section{Preliminaries and the Proof of Theorem \ref{Main1}}

Let $M = \mathbb{S}^{2} \backslash \left\{N,S\right\}$ be the standard Euclidean sphere in which two antipodal points $N$ and $S$ are removed. Denote by $g$ the usual metric of $\mathbb{S}^2$ induced from $\mathbb{R}^3$, and by $\nabla$ the Levi-Civita connection associated to $g$. Consider the oriented orthonormal local frame $\left\{ e_1, e_2\right\}$ on $M$, where $e_1$ is tangent to the meridians and $e_2$ to the parallels. Let $\vec{v}$ be a unit vector field tangent to $M$ and consider another oriented orthonormal local frame $\left\{u_1 = \vec{v}^{\perp}, u_2 = \vec{v}\right\}$ on $M$ and its dual basis  $\left\{ \omega_1, \omega_2 \right\}$ compatible with the orientation of $\left\{ e_1, e_2\right\}$.

In dimension $2$, the volume of $\vec{v}$ is given by
\begin{eqnarray}\label{volreduces}
\vol(\vec{v}) = \int_{\mathbb{S}^2}{\sqrt{1+ \gamma^2 + \delta^2}}\nu,
\end{eqnarray}
where $\gamma=g(\nabla_{\vec{v}}\vec{v}, \vec{v}^{\perp})$ and $\delta = g(\nabla_{v^{\perp}}\vec{v}^{\perp}, \vec{v})$ are the geodesic curvatures associated to $\vec{v}$ and $\vec{v}^{\perp}$, respectively.

Let $\mathbb{S}^1_{\alpha}$ be the parallel of $\mathbb{S}^2$ at latitude $\alpha \in (-\frac{\pi}{2}, \frac{\pi}{2})$ and $\mathbb{S}^1_{\beta}$ be the meridian of $\mathbb{S}^2$ at longitude $\beta \in (0, 2\pi)$. 

\begin{prop}\label{integrando_volume} 
Let $\theta \in [0,\pi/2]$ be the oriented angle from $e_2$ to $\vec{v}$. If $\vec{v}=(\cos \theta) e_1 + (\sin \theta) e_2$ and $\vec{v}^{\perp}=(-\sin \theta)e_1 + (\cos \theta) e_2$, then
\[1 + \gamma^2 + \delta^2 = 1 + (\tan \alpha + \theta_1)^2 + \theta_2^2,\]
where $\theta_1 = d\theta(e_1)$, $\theta_2 = d\theta(e_2)$.
\end{prop}
\begin{proof} 
We have
\begin{equation}\label{def_gamma}
\gamma	= g\big( \nabla_{\vec{v}}\vec{v}, \vec{v}^{\perp}\big)
				= g\big( \nabla_{(\cos \theta) e_1 + (\sin \theta) e_2} \left[(\cos \theta) e_1 + (\sin \theta) e_2 \right], \vec{v}^{\perp}  \big) \end{equation}
\[					
				= g\big( \nabla_{(\cos \theta) e_1} (\cos \theta) e_1, \vec{v}^{\perp}   \big) 
				+ g\big( \nabla_{(\sin \theta) e_2} (\cos \theta) e_1, \vec{v}^{\perp} \big) 
				+  g\big( \nabla_{(\cos \theta) e_1} (\sin \theta)e_2, \vec{v}^{\perp} \big)
				+ g\big( \nabla_{(\sin \theta) e_2} (\sin \theta) e_2, \vec{v}^{\perp} \big)
\]
and
\begin{equation}\label{def_delta}
\delta	= g\big( \nabla_{\vec{v}^{\perp}}\vec{v}^{\perp}, \vec{v}\big)
				= g\big( \nabla_{(-\sin \theta) e_1 + (\cos \theta) e_2} \left[-(\sin \theta) e_1 + (\cos \theta) e_2 \right], \vec{v}  \big) \end{equation}
\[					
				= g\big( \nabla_{(-\sin \theta) e_1} (-\sin \theta) e_1, \vec{v}   \big) 
				+ g\big( \nabla_{(\cos \theta) e_2} (-\sin \theta) e_1, \vec{v} \big) 
				+  g\big( \nabla_{(-\sin \theta) e_1} (\cos \theta)e_2, \vec{v} \big)
				+ g\big( \nabla_{(\cos \theta) e_2} (\cos \theta) e_2, \vec{v} \big)
\]
We write $\gamma$ and $\delta$ as the following sums
\[\gamma = A + B + C + D \enspace\text{and}\enspace \delta = A^{\prime} + B^{\prime} + C^{\prime} + D^{\prime},\]
with
\[ A = g\big( \nabla_{(\cos \theta) e_1} (\cos \theta) e_1, \vec{v}^{\perp}  \big), \hspace{0.2cm}B  = g\big( \nabla_{(\sin \theta) e_2} (\cos \theta) e_1, \vec{v}^{\perp} \big),\]
\[C = g\big( \nabla_{(\cos \theta) e_1} (\sin \theta)e_2, \vec{v}^{\perp} \big) \hspace{0.2cm}\mbox{and}\hspace{0.2cm}  D = g\big( \nabla_{(\sin \theta) e_2} (\sin \theta) e_2, \vec{v}^{\perp} \big) \]
and
\[ A^{\prime} = g\big( \nabla_{(-\sin \theta) e_1} (-\sin \theta) e_1, \vec{v}   \big) \hspace{0.2cm},\hspace{0.2cm} B^{\prime}  = g\big( \nabla_{(\cos \theta) e_2} (-\sin \theta) e_1, \vec{v} \big) \]
\[ C^{\prime} = g\big( \nabla_{(-\sin \theta) e_1} (\cos \theta)e_2, \vec{v} \big) \hspace{0.2cm}\mbox{and}\hspace{0.2cm}  D^{\prime} = g\big( \nabla_{(\cos \theta) e_2} (\cos \theta) e_2, \vec{v} \big). \]

First observe that $\tan \alpha = g\big( \nabla_{e_1}e_1, e_2 \big)$ and $\nabla_{e_2}e_2 = 0$. By an elementary computation 
we obtain 
\[A = \sin^2\theta (\cos \theta) \theta_1 + \cos^3\theta \tan \alpha, \quad B = (\sin^3 \theta) \theta_2,\]		
\[C = (\cos^3\theta)\theta_1 + \sin^2 \theta \cos \theta \tan \alpha\enspace\mbox{and}\enspace D = (\sin \theta \cos^2) \theta_2 \]
and
\[A^{\prime} = (\sin \theta \cos^2 \theta) \theta_1 + \sin^3\theta \tan \alpha, \quad B^{\prime} = (-\cos^3 \theta) \theta_2. \]
\[C^{\prime} = (\sin^3\theta)\theta_1 + \sin \theta \cos^2 \theta \tan \alpha\enspace \text{and}\enspace D^{\prime} = -(\sin^2 \theta \cos \theta) \theta_2.\] 
Moreover,
\[
\gamma	= (\cos^3\theta \tan \alpha + \sin^2 \theta \cos \theta \tan \alpha) + (\sin^2\theta (\cos \theta) \theta_1 + (\cos^3\theta)\theta_1 ) + ((\sin^3 \theta) \theta_2 + (\sin \theta \cos^2) \theta_2)\]
\[
= \cos \theta \tan \alpha + (\cos \theta) \theta_1 + (\sin \theta) \theta_2 = \cos \theta (\tan \alpha + \theta_1) + (\sin \theta) \theta_2 
\]
and
\[
	\delta =  (\sin^3\theta \tan \alpha +  \sin \theta \cos^2 \theta \tan \alpha) + ((\sin \theta \cos^2 \theta) \theta_1 + (\sin^3\theta)\theta_1) + ((-\cos^3 \theta) \theta_2-(\sin^2 \theta \cos \theta) \theta_2). \]
\[ = \sin \theta \tan \alpha + (\sin \theta) \theta_1 - (\cos \theta) \theta_2 = \sin \theta (\tan \alpha + \theta_1) - (\cos \theta) \theta_2. 
\]
Finally,
\begin{equation}\label{curva_geode_tangente}
\gamma = \cos \theta (\tan \alpha + \theta_1) + (\sin \theta) \theta_2 
\end{equation}
\begin{equation}\label{curva_geode_ortogonal}
\delta = \sin \theta (\tan \alpha + \theta_1) - (\cos \theta) \theta_2.
\end{equation}
From equations (\ref{curva_geode_tangente}) and (\ref{curva_geode_ortogonal}), we find
\[ 1 + \gamma^2 + \delta^2 = 1 + \left(\cos \theta (\tan \alpha + \theta_1) + (\sin \theta) \theta_2 \right)^2 +  \left( \sin \theta (\tan \alpha + \theta_1) - (\cos \theta) \theta_2 \right)^2 \]
\[
= 1 + \cos^2 \theta (\tan \alpha + \theta_1)^2 + (\sin \theta)^2 \theta_2^2 + \sin^2 \theta (\tan \alpha + \theta_1)^2 + (\cos^2 \theta) \theta_2^2 = 1 + (\tan \alpha + \theta_1)^2 + \theta_2^2.
\]
Therefore,
\[
1 + \gamma^2 + \delta^2 = 1 + (\tan \alpha + \theta_1)^2 + \theta_2^2.
\]
\end{proof}


Proposition \ref{integrando_volume} allows us to rewrite the volume functional as an integral depending on $\alpha$ and the derivatives of $\theta$
\begin{equation}\label{vol_usando_novo_integrando}
\vol(\vec{v})= \int_{M}{\sqrt{1+ \left( \tan \alpha + \theta_1\right)^2 + \theta_2^2}}.
\end{equation}

\begin{proof}[Proof of Theorem \ref{Main1}]
Given $\varphi$ such that $0 \leq \varphi \leq 2\pi$,
\[ 
1+ \left( \tan \alpha + \theta_1\right)^2 + \theta_2^2 =
\]
\[
= \left( \cos\varphi + \sin\varphi \left[ \sqrt{\left(\tan \alpha + \theta_1\right)^2 + \theta_2^2}\right] \right)^2
+ 
\left( -\sin\varphi + \cos \varphi \left[ \sqrt{\left(\tan \alpha + \theta_1\right)^2 + \theta_2^2}\right] \right)^2.
\]
Hence,
\[
\vol(\vec{v}) = \int_{M}{\sqrt{\left( \cos\varphi + \sin\varphi \left[ \sqrt{\left(\tan \alpha + \theta_1\right)^2 + \theta_2^2}\right] \right)^2
+ 
\left( -\sin\varphi + \cos \varphi \left[ \sqrt{\left(\tan \alpha + \theta_1\right)^2 + \theta_2^2}\right] \right)^2}}.
\]
Remember that
\[ 
1+ \left( \tan \alpha + \theta_1\right)^2 + \theta_2^2 \geq 
1+ \left( \tan \alpha + \theta_1\right)^2,
\]
implies
\[ 
\sqrt{1+ \left( \tan \alpha + \theta_1\right)^2 + \theta_2^2} \geq 
\sqrt{1+ \left( \tan \alpha + \theta_1\right)^2}.
\]
From the general inequality, $\sqrt{a^2 + b^2} \geq |a\cos \varphi + b \sin \varphi|$, for any $a$, $b$, $\varphi \in \mathbb{R}$, we have
\[
\sqrt{1+ \left( \tan \alpha + \theta_1\right)^2} \geq \left|\cos \varphi + \sin \varphi(\tan(\alpha)+\theta_1) \right|. 
\]
Therefore,
\[
\vol(\vec{v}) = \int_{M}{\sqrt{\big( \cos\varphi + \sin\varphi |\tan \alpha + \theta_1|\big)^2 + \big(-\sin\varphi + \cos \varphi |\tan \alpha + \theta_1| \big)^2}}.
\]
\begin{equation}\label{vol_theta2_zero}
\geq \int_{M}{\cos \varphi + \sin \varphi \left| \tan \alpha + \theta_1 \right|}
\end{equation}
This inequality is valid for all $\varphi$ such that $0 \leq \varphi \leq 2\pi$. \\
As a next step one consider the following conditions:
\begin{enumerate}[i)]
\item 
\[\varphi_k(\alpha) = \arctan \left( \tan\alpha + \frac{k-1}{\cos \alpha}\right);\]
\item
\[\tan\big(\varphi_k(\alpha)\big) = \tan\alpha + \frac{k-1}{\cos \alpha}.\]
\end{enumerate}
Replacing these conditions in equation (\ref{vol_theta2_zero}) we find
\begin{equation}\label{volume_com_alfa}
\vol(\vec{v}) \geq \int_{M}{ \left(\cos \big( \varphi_k(\alpha)\big) + \sin \big(\varphi_k(\alpha) \big)\right) \left| \tan \alpha + \theta_1 \right| \nu}.
\end{equation}
Condition i) provides that 
\[
\cos \big(\varphi_k(\alpha)\big) = 
\frac{\cos\alpha}{\sqrt{1 + (k-1)^2 + 2(k-1)\sin\alpha}}, \hspace{0.2cm} - \frac{\pi}{2} \leq \alpha \leq \frac{\pi}{2} 
\]
\[
\sin \big(\varphi_k(\alpha)\big) = \frac{ k-1 + \sin\alpha}{\sqrt{1 + (k-1)^2 + 2(k-1)\sin\alpha}}, \hspace{0.2cm} - \frac{\pi}{2} \leq \alpha \leq \frac{\pi}{2}. 
\]
Thus, we rewrite equation (\ref{volume_com_alfa}) as
\begin{equation}\label{***}
\lim_{\alpha_0 \to - \frac{\pi}{2}} \left[\int_{\alpha_0}^{\frac{\pi}{2}}\int_0^{\frac{\pi}{2}}\left( \frac{\cos\alpha}{\sqrt{1 + (k-1)^2 + 2(k-1)\sin\alpha}} + \frac{ k-1 + \sin\alpha}{\sqrt{1 + (k-1)^2 + 2(k-1)\sin\alpha}} \left| \tan \alpha + \theta_1 \right|\right) \right]d\beta d\alpha.
\end{equation}
We remember that the connection form $w_ {12}$ is given by
\[
\omega_{12} = \delta \omega_1 + \gamma \omega_2,
\]
where $\{ \omega_1, \omega_2 \}$ is dual basis of $\{\vec{u_1}, \vec{u_2} \}$. Let $i^* : \mathbb{S}^1_{\alpha} \hookrightarrow \mathbb{S}^2$ be the inclusion map, and $e_1= \sin \theta \vec{v}^{\perp} + \cos\theta \vec{v}$, then
\[i^*(\omega_{12})(e_1) = \delta \sin\theta + \gamma\cos \theta.\]
From equations (\ref{curva_geode_tangente}) and (\ref{curva_geode_ortogonal}),
we have
\begin{eqnarray}\nonumber
i^*(\omega_{12}) &=& \sin\theta \left[\sin \theta \big(\tan(\alpha)+\theta_1\big)-\cos\theta\theta_2\right] + \cos \theta \left[\cos\theta \big(\tan(\alpha)+\theta_1\big) +\sin\theta \theta_2\right]\\ \nonumber
& = &  \tan(\alpha) + \theta_1.
\end{eqnarray}

Thus, from (\ref{***}) 
\begin{equation}\label{integralcomlimitemenor}
\vol(\vec{v})
\geq \lim_{\alpha_0 \to -\frac{\pi}{2}} \left( \int_{\alpha_0}^{\frac{\pi}{2}} \left( \int_{0}^{2\pi}\left(\frac{\cos \alpha + \big((k-1) + \sin \alpha\big)i^*\omega_{12}}{\sqrt{1+(k-1)^2 + 2(k-1)\sin \alpha}}\right) d\beta\right) d\alpha\right).
\end{equation}
To compute the integral of $i^*\omega_{12}$ with domain in the parallel of $\mathbb{S}^2$ at constant latitude $\alpha$, we follow the same arguments of Brito, Chac\'on and Johnson in the proof Theorem 1.1 of \cite{BCJ}. 
Denote by $\omega$ the connection form $\omega_{12}$ and 
\[
\mathbb{S}^2_{\alpha} = \{(x,y,z \in \mathbb{R}^3 ; z \geq \sin \alpha, \alpha_0 \leq \alpha \leq \frac{\pi}{2})\}.
\]
The  $2$-form $d\omega$ is given by
\[
d\omega = \omega_1 \wedge \omega_2.
\]
A simple application of Stokes' theorem implies that
\[
\int_{\mathbb{S}^2_{\alpha}}d\omega = 2\pi \big(I_{N}(\vec{v}) \big)- \int_{\mathbb{S}^1_{\alpha}} i^*\omega_{12}d\beta.
\]
Suppose that $I_{N}(\vec{v}) = \sup \{I_{N}(\vec{v}), I_{S}(\vec{v})\} = k$, we obtain
\begin{equation}\label{integral_pullback}
\int_{\mathbb{S}^1_{\alpha}}i^*\omega_{12}d\beta = 2\pi k - \mbox{Area}\big( \mathbb{S}^2_{\alpha}\big) = 2\pi k - 2\pi\big( 1 - \sin \alpha \big) = 2\pi\big( k - 1 + \sin \alpha \big).
\end{equation}
From equation (\ref{integralcomlimitemenor}), 
\[\vol(\vec{v})
\geq \lim_{\alpha_0 \to -\frac{\pi}{2}} \left( \int_{\alpha_0}^{\frac{\pi}{2}} \left( \int_{0}^{2\pi}\left(\frac{\cos \alpha + \big((k-1) + \sin \alpha\big)i^*\omega_{12}}{\sqrt{1+(k-1)^2 + 2(k-1)\sin \alpha}}\right) d\beta\right) d\alpha\right)
\]
\[
\geq \lim_{\alpha_0 \to -\frac{\pi}{2}}  \int_{\alpha_0}^{\frac{\pi}{2}} \left( \frac{\cos \alpha}{\sqrt{1+(k-1)^2 + 2(k-1)\sin \alpha}}\int_{0}^{2\pi}d\beta
 + \frac{\big((k-1) + \sin \alpha\big)}{\sqrt{1+(k-1)^2 + 2(k-1)\sin \alpha}}  \int_{\mathbb{S}^1_{\alpha}}i^*\omega_{12}d\beta \right) d\alpha
\]
\[
\geq \lim_{\alpha_0 \to -\frac{\pi}{2}}  \int_{\alpha_0}^{\frac{\pi}{2}} \left( \frac{2\pi\cos^2 \alpha}{\sqrt{1+(k-1)^2 + 2(k-1)\sin \alpha}} 
 + \frac{2\pi\big((k-1) + \sin \alpha\big)^2}{\sqrt{1+(k-1)^2 + 2(k-1)\sin \alpha}}\right) d\alpha,
\]
where the last inequality is obtained from (\ref{integral_pullback}). Therefore,
\[
\vol(\vec{v})
\geq 2 \pi \lim_{\alpha_0 \to -\frac{\pi}{2}}  \int_{\alpha_0}^{\frac{\pi}{2}} \left(\frac{\cos^2 \alpha + \big((k-1) + \sin \alpha\big)^2}{\sqrt{1+(k-1)^2 + 2(k-1)\sin \alpha}}\right)d\alpha.
\]
Analogously,
\[
\vol(\vec{v})
\geq 2 \pi \lim_{\alpha_0 \to -\frac{\pi}{2}}  \int_{\alpha_0}^{\frac{\pi}{2}}\left({\sqrt{1+(k-1)^2 + 2(k-1)\sin \alpha}}\right)d\alpha.
\]
An elementary computation gives us
\[
\vol(\vec{v})
\geq 2 \pi \int_{-\frac{\pi}{2}}^{\frac{\pi}{2}}{\sqrt{(k-2)^2 + 4(k-1)\sin^2 \left(\frac{\alpha}{2} + \frac{\pi}{4} \right)}}d\alpha.
\]

Assume that $t = \frac{\alpha}{2} + \frac{\pi}{4}$, then 
\begin{equation}\label{volumenavariavelt}
\vol(\vec{v})
\geq 2 \pi \int_{0}^{\frac{\pi}{2}}{\sqrt{(k-2)^2 + 4(k-1)\sin^2 t}}dt.
\end{equation}
Consider $k > 2$ and an ellipse $\xi_{k}$ given by 
\[\frac{x^2}{k^2} + \frac{y^2}{(k-2)^2} = 1.\]
Let $\mu$ be a parametrization for $\xi_{k}$ defined by
$\mu(t) = (k \cos t, (k-2)\sin t)$.
Its length is
\begin{equation}\label{compr_elipse}
L(\xi_k) = \int_{0}^{2\pi}\left(\sqrt{ (k-2)^2 + 4(k-1)\sin^2 t} \right)dt.
\end{equation}
We replace the value found in (\ref{compr_elipse}) in inequality (\ref{volumenavariavelt}) and conclude
\[
\vol(\vec{v})
\geq 2 \pi \frac{1}{2}L(\xi_k) = \pi L(\xi_k).
\]
\end{proof}

\section{Area-minimizing vector fields $\vec{v}_k$ on $M$}

Consider the oriented orthonormal global frame $\left\{ \vec{e}_1, \vec{e}_2 \right\}$ on $\mathbb{S}^2\backslash\left\{ N,S\right\}$  such that
\[ \vec{e}_1(p) = \frac{1}{\sqrt{x^2 +y^2}}(-y,x,0)\]
\[ \vec{e}_2(p) = \left(\frac{xz}{\sqrt{x^2 +y^2}},\frac{yz}{\sqrt{x^2 +y^2}},- \sqrt{x^2 + y^2} \right),\]

where $p = (x,y,z) \in \mathbb{S}^2\backslash\left\{ N,S\right\}$.

\begin{defi} Let $p \in \mathbb{S}^2 \backslash \left\{N, S \right\}$. Let $k$ a positive integer such that $k \geq 1$. We define
\begin{enumerate}
\item If $k=1$, then $\vec{v}_1 = \vec{e}_2$;
\item If $k \geq 2$, then $\vec{v}_k(p) = \cos \theta(p)\vec{e}_1(p) + \sin \theta(p)\vec{e}_2(p)$, where $
\theta: \mathbb{S}^2 \backslash \left\{N, S \right\} \to \mathbb{R}$ satisfy 
\[\nabla_{\vec{e}_1(p)}\theta (p) = \frac{k-1}{\sqrt{x^2 + y^2}}\]
i.e., $\theta$ has constant variation along the parallel $x^2 + y^2 = \alpha$, with $\alpha$ constant.
\end{enumerate}
\end{defi}
\begin{center}
\includegraphics[height=8cm]{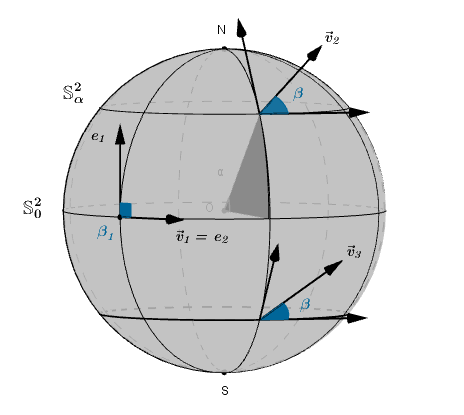}
\end{center}

If we use spherical coordinates $(\beta, \alpha)$, so that $p=(\cos\alpha\cos\beta, \cos\alpha\sin\beta, \sin\alpha)$, we can say that the vector $\vec{v}_k$ spins at a constant speed of rotation along the parallel $\alpha$, where $\alpha$ is constant. Moreover, $\vec{v}$ gives exactly $k-1$ turns when it passes the $\alpha$ parallel, with respect to the referential $\left\{\vec{e}_1, \vec{e} _2 \right\} $, and it gives $k$ turns with respect to a fixed polar referential, in this case,
\[
\nabla_{\vec{e}_1(p)} \theta(p) = \frac{k-1}{|\cos \alpha|}.
\]

\section{Proof of Theorem \ref{Main2}}

First step to proof Theorem \ref{Main2} is to rewrite the vector field $\vec{v}_k$ with respect to a global orthonormal special frame on $\S^2\backslash\{N,S\}$.\\
Let $P := P_{\gamma}:T_{\gamma(s_0)}\S^2 \longrightarrow T_{\gamma(s)}\S^2$ be the parallel transport, where $\gamma$ 
is the geodesic curve that connects the north pole with a point $p \in \S^2$.
Consider the canonical identification $T_{\gamma(s_0)}\S^2 \approx \mathbb{R}^2$. Assume that $\gamma(s_0) = (0,0,1) = N$, $\gamma(s) = p \in \S^2$, for some $s \in I$, and
\[\vec{u_1}(p) = P_{\gamma} (\vec{e}_1) \enspace\text{e}\enspace \vec{u_2}(p) = P_{\gamma} (\vec{e}_2),\]
such that $\vec{u_1}(1,0,0) = \vec{e_1}(1,0,0)$ and $\vec{u_2}(1,0,0) = \vec{e_2}(1,0,0)$. Then,
\begin{equation}\label{campo_miniz_ref_modificado}
\vec{v_k}(p) =  \cos(kt)\vec{u_1}(p) + \sin(kt)\vec{u_2}(p).
\end{equation}
In terms of $\{\vec{e_1}, \vec{e_2}\}$
\[\vec{v_k}(p) =  \cos\big((k-1)t\big)\vec{e_1}(p) + \sin\big((k-1)t\big)\vec{e_2}(p).\]
Consider $\mathbb{S}^2 = \mathbb{S}^2_+ \cup \mathbb{S}^2_-$, where $\mathbb{S}^{2}_{+}$ and $\mathbb{S}^{2}_{-}$ are the northern and southern hemisphere, respectively. Let $T_1 \S^2$ be unit tangent bundle. We denote $\pi_{+}$ the canonical projection of the $T_1\left(\S^2_{+}\right)$ from to $\mathbb{S}^2_+$. Let $\varphi$ be a global trivialization on $T_1\S^2_{+}$ given by
\begin{equation}\label{trivi_global}
\varphi: \S^2_{+}\times \S^1 \to \pi_{+}^{-1} \left(\S^2_{+} \right),\enspace  (p,\sigma) \mapsto (p, P_{\gamma}(\sigma)),\end{equation}
In this way, we have the commutative diagram
\[
\xymatrix{\mathbb{S}^{2}_{+}\times \S^{1} \ar[rr]^{\varphi} \ar[rd]_{p_1} & & \pi_{+}^{-1}(\mathbb{S}^{2}_{+}) \ar[dl]^{\pi_{+}} \\ & \mathbb{S}^{2}_{+} & }
\]
where $p_1: \mathbb{S}^{2}_{+}\times \S^{1} \to \mathbb{S}^{2}_{+}$ is the projection map. 
Let $D_{\pi/2} = \{ X \in T_{s(0)}\S^2_{+} \cong \mathbb{R}^2 ~:~ X \in B_{\pi/2}(0) \}$ be the open disk centered at the origin with radius $\pi/2$. Notice that $\exp\big(D_{\pi/2}\big) = \S^2_{+}$ and consider $D_{\pi/2} \times \S^1$ 
as a subset of $\mathbb{R}^4$. We define
\[ \psi: D_{\pi/2} \times \S^1  \to \S^2_{+}\times \S^1, \quad (X, \sigma) \mapsto (\exp(X), \sigma)\,\]
where $\exp : T_1 \S^2_{+} \to \S^2_{+}$. 
Composing this map $\psi$ with the global trivialization $\varphi$ we get a map 
\begin{equation}\label{composition}
\varphi\circ\psi: {D_{\pi/2} \times \S^1 \to \pi_{+}}^{-1}(\mathbb{S}^{2}_{+}), \quad (X, \sigma) \mapsto \left(\exp(X), P_{\exp(X)}(\sigma)\right).
\end{equation}
Let $\vec{v}_k$ be an unit vector field with the singularity $N$ of Poincar\'e index $k$, where $k$ is an even positive integer. We associate to the unit vector field $\vec{v}_k$ the surface $\mathcal{M}_k \subset D_{\pi/2} \times \S^1 \subset \mathbb{R}^4$ given by
\[\mathcal{M}_k := \left\{\big(r \cos(t), r\sin(t), \cos(kt), \sin(kt)\big) \in \mathbb{R}^4 ~:~ \enspace 0 \leq t \leq 2\pi\enspace\text{e}\enspace 0 \leq r \leq \pi/2 \right\}.\]

\begin{prop}\label{sup_regrada}
The surface $\mathcal{M}_k$ is a ruled surface of $\mathbb{R}^4$.
\end{prop}

\begin{proof} 
Define
\[G_t : = \left\{ \left(-\frac{\pi}{2}\cos(t), -\frac{\pi}{2}\sin(t), \cos(kt), \sin(kt) \right) \in \mathbb{R}^4 ~:~ -\pi \leq t \leq \pi \right\} \]
\[H_t : = \left\{  \left(\frac{\pi}{2}\cos(t), \frac{\pi}{2}\sin(t), \cos(kt), \sin(kt) \right) \in \mathbb{R}^4 ~:~ -\pi \leq t \leq \pi \right\}  \]
\[O_t : = \left\{  \left(0, 0, \cos(kt), \sin(kt) \right) \in \mathbb{R}^4 ~:~ -\pi \leq t \leq \pi \right\},  \]
we are going to show that 
\[\mathcal{M}_k = \overline{G_tO_t} \bigcup \overline{O_tH_t}. \]

Consider $q \in \mathcal{M}_k$, $q_1 \in \overline{G_tO_t}$ and $q_2 \in \overline{O_tH_t}$. If $\alpha := -\frac{2 r}{ \pi}$ and $\beta := \frac{2 r}{ \pi}$, then $q$ can be written as $\alpha q_1$ or as $\beta q_2$. It provides us $q \in \overline{G_tO_t} \bigcup \overline{O_tH_t}$. So, $\mathcal{M}_k \subset \overline{G_tO_t} \bigcup \overline{O_tH_t}$. 

Now, we show that $\overline{G_tO_t}$ and $\overline{O_tH_t}$ are subsets of $\mathcal{M}_k$.
A point $q_1$ belongs to $\overline{G_tO_t}$ if, and only if, there exist $\alpha \in [0,1]$ such that
\[ q_1 = \alpha\left( -\frac{\pi}{2} \cos(t), -\frac{\pi}{2}\sin(t), \cos(kt), \sin(kt)\right) +(1-\alpha) \left(0, 0, \cos(kt), \sin(kt)\right)\]
We observe that
\begin{equation}\label{eq1}
\left(-\frac{\pi}{2}\cos(t), -\frac{\pi}{2}\sin(t)\right) = \left(\frac{\pi}{2}\cos(\pi + t), \frac{\pi}{2}\sin(\pi + t)\right),
\end{equation}
and since $k=2k'$, for some $k' \in \mathbb{Z}$, we have 
\begin{equation}\label{eq2}
\left(\cos\big(k(\pi + t)\big), \sin\big(k(\pi + t)\big)\right) = \left(\cos(2k'\pi + kt), \sin(2k'\pi + kt)\right) = \left(\cos(kt), \sin(kt)\right).
\end{equation}
From equations (\ref{eq1}) and (\ref{eq2}) we conclude
\[ q_1 = \left( \frac{\alpha\pi}{2}\cos(\pi + t), \frac{\alpha\pi}{2}\sin(\pi + t), \cos\big(k(\pi + t)\big) , \sin\big(k(\pi + t)\big)\right). \]

Notice that $\alpha$ satisfies
\[ 0 \leq \alpha \leq 1 \enspace\Rightarrow\enspace 0 \leq \alpha\frac{\pi}{2} \leq 1, \enspace r^{\prime}:= \alpha \frac{\pi}{2}\enspace\Rightarrow\enspace  0 \leq r^{\prime} \leq \frac{\pi}{2}.\]
Thus, we rewrite $q_1$ as
\[q_1 = \left( r^{\prime}\cos(\pi + t), r^{\prime}\sin(\pi + t), \cos\big(k(\pi + t)\big) , \sin\big(k(\pi + t)\big)\right). \]

Assume that $t' = \pi + t$, then 
\[-\pi \leq t \leq \pi \enspace\Rightarrow\enspace -\pi \leq t' - \pi \leq \pi \enspace\Rightarrow\enspace 0 \leq t^{\prime} \leq 2\pi. \]Therefore,
\[q_1 = \left(r^{\prime}\cos(t^{\prime}), r^{\prime}\sin(t^{\prime}), \cos(kt^{\prime}) , \sin(kt^{\prime})\right) \in \mathcal{M}_k, \]
so we conclude that $\in\overline{G_tO_t} \subset \mathcal{M}_k$.

Observe that a point $q_2$ belongs to $\overline{O_tH_t}$ if, and only if, there exist $\beta \in [0,1]$ such that 
\[q_2 = (1-\beta)\left( \frac{\pi}{2} \cos(t), \frac{\pi}{2}\sin(t), \cos\big(k(\pi + t)\big), \sin\big(k(\pi + t)\big)\right) +\beta \left(0, 0, \cos\big(k(\pi + t)\big), \sin\big(k(\pi + t)\big)\right).\]
Therefore, an analogously computation shows that $q_2 \in \mathcal{M}_k$. Thereby, $\overline{G_tO_t} \bigcup \overline{O_tH_t} \subset \mathcal{M}_k$. 
\end{proof}
If $\alpha=1/2$, $\beta = -1/2$ and $t=0$, then $q_1=q_2 = \left(-\frac{\pi}{4} , 0, 1, 0 \right)$. Therefore, this union is not disjoint union. 

\begin{prop}
The surface $\mathcal{M}_k$ is the image of a smooth immerion of Moebius band.
\end{prop}
\begin{proof}
Define $j: \left[-\frac{\pi}{2}, \frac{\pi}{2}\right] \times \left[0, 2k\pi\right] \to D_{\frac{\pi}{2}} \times \S^1$ by
\[j(r,t) = (r \cos(t), r\sin(t), \cos(kt), \sin(kt)).\]
The map $j$ is an immersion of class $C^{\infty}$ in which $\mbox{Im}(j)$ is Moebius band. 
From definition of $\mathcal{M}_k$ we have $\mathcal{M}_k = \mbox{Im}(j).$
\end{proof}



Denote by $\overline {\vec{v_k} (\S^{2}_{+} \backslash \{N\})}$,\enspace$\overline {\vec{v_k} (\S^{2}_{-} \backslash \{S\})}$ and $\overline {\vec{v_k} (\S^{2} \backslash \{N, S\})}$ the topological closures of $\vec{v_k}$ on $\S^{2}_{+} \backslash \{N\}$, $\S^{2}_{-} \backslash \{S\}$ and $\S^{2} \backslash \{N,S\}$, respectively.

\begin{prop}\label{indice_par} If $k$ is even, then $\overline {\vec{v_k} (S^{2}_{+} \backslash \{N \})}$ is the image of the immersion $\varphi \circ \psi$ of the Moebius band $\mathcal{M}_k$ in $T_1 \S^{2}$.
\end{prop}

\begin{proof} 
The topological closure of $\vec{v}_k$ is the image of the global trivialization given by (\ref{trivi_global}). Thereby, it remains to prove that $\varphi \circ \psi\left(\mathcal{M}_k\right) \subset \overline {\vec{v_k} (\mathbb{S}^{2}_{+} \setminus \{N \})}.$ Given $(x_1,x_2,x_3,x_4)\in \mathcal{M}_k$, where $(x_1,x_2) = (r\cos(t), r\sin(t)) \in D_{\pi/2}$ and $(x_3,x_4) = (\cos(kt), \sin(kt)) \in \S^1$ we have,
\[ \varphi \circ \psi(x_1,x_2,x_3,x_4) = \left( \exp_{0}(r\cos(t), r\sin(t)), P_{\exp_{0}}(\cos(kt), \sin(kt)) \right). \]
Consider $\gamma(r) \in T_{\gamma(s)\mathbb{S}^2_+}$. From equation (\ref{campo_miniz_ref_modificado}) we obtain
\[ \left( \exp_{0}(r\cos(t), r\sin(t)), P_{\exp_{0}}(\cos(kt), \sin(kt)) \right) = \left( \gamma(r), \cos(kt)P(e_1) + \sin(kt)P(e_2) \right)\]
\[ = \left( p, \cos(kt)u_1 + \sin(kt)u_2\right) = (p, \vec{v}_k).\]
Thus, $\varphi \circ \psi(x_1,x_2,x_3,x_4) = (p, \vec{v}_k) \in \overline {\vec{v_k} (S^{2}_{+} \setminus \{N \})}$. Therefore, $\varphi \circ \psi\left(\mathcal{M}_k\right) \subset \overline {\vec{v_k} (\mathbb{S}^{2}_{+} \setminus \{N \})}$.
\end{proof}

\begin{prop}\label{indice_impar} If $k$ is even, then  $\overline{\vec{v_k} (\mathbb{S}^{2}_{-} \setminus \{S \})}$ is the image of the immersion $\varphi \circ \psi$ of the Moebius band $\mathcal{M}_k$ in $T_1 \S^{2}$.
\end{prop}
Assume that $\vec{v}_k$ an unit vector field such that $I_{\vec{v}_k}(N)$ is positive and $I_{\vec{v}_k}(S)$ is negative, which the both are even interges numbers.

\begin{proof}[Proof of Theorem \ref{Main2}] 
An immersion of class $C^{\infty}$ of the Klein bottle in $T_1\S^2$ is obtained by gluing two images of the Moebius band along the boundary given by Propositions \ref{indice_par} and \ref{indice_impar}. As $\vec{v}_k$ is area-minimizing vector field in its topological conjugation class, it follows that the section seen as surface in $T_1\S^2$ is geometrically minimal, i.e., its has zero mean curvature. Therefore, the topological closure of $\vec{v_k}$ is a minimal surface in $T_1\S^2$.

\begin{center}
\includegraphics[height=6cm]{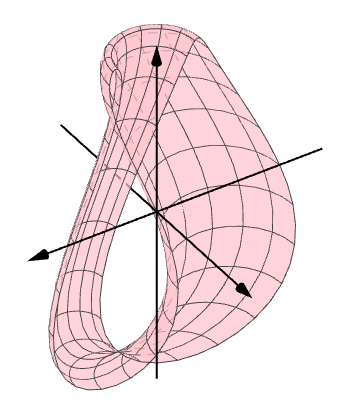}
\includegraphics[height=6cm]{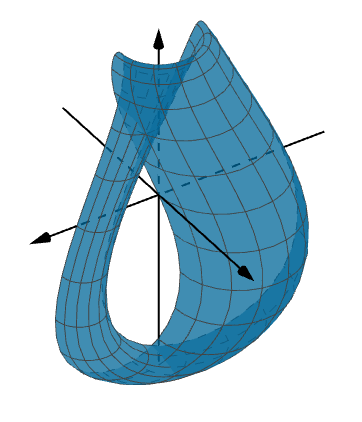}
\includegraphics[height=6.2cm]{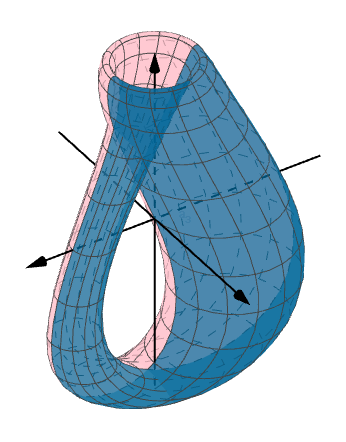}
\end{center}

\end{proof}

\end{document}